\documentclass{article}

\usepackage[T1]{fontenc}
\usepackage[utf8]{inputenc}
\usepackage{lmodern}
\usepackage{geometry}
\usepackage{amsmath,amssymb,amsfonts,amsthm,mathtools}
\usepackage{mathrsfs}
\usepackage{enumitem}
\usepackage{color}
\usepackage{microtype}
\usepackage{hyperref}
\usepackage[nameinlink,noabbrev]{cleveref}

\usepackage{indentfirst} 
\newtheorem{theorem}{Theorem}[section]
\newtheorem{lemma}[theorem]{Lemma}
\newtheorem{proposition}[theorem]{Proposition}
\newtheorem{corollary}[theorem]{Corollary}
\newtheorem{conjecture}[theorem]{Conjecture}

\newtheorem{observation}[theorem]{Observation}
\theoremstyle{definition}
\newtheorem{definition}[theorem]{Definition}
\theoremstyle{remark}

\newcommand{\R}{\mathbb R}
\newcommand{\N}{\mathbb N}
\newcommand{\cC}{\mathcal C}
\newcommand{\cF}{\mathcal F}
\newcommand{\cG}{\mathcal G}

\newcommand{\cA}{\mathcal A}
\newcommand{\cB}{\mathcal B}
\newcommand{\cD}{\mathcal D}
\newcommand{\lk}{\operatorname{lk}}
\newcommand{\Span}{\operatorname{span}}

\newcommand{\rank}{\operatorname{rank}}

\newcommand{\set}[1]{\left\{#1\right\}}
\newcommand{\abs}[1]{\left|#1\right|}
\newcommand{\paren}[1]{\left(#1\right)}

\newcommand{\binomset}[2]{\binom{#1}{#2}}
\newcommand{\col}{\operatorname{col}}

\title{Dimensional colorful Helly theorems and topological variants}
\author{
Wei Rao\thanks{Moscow Institute of Physics and Technology, Institutsky Lane 9, Dolgoprudny, Moscow Region, 141700, Russia. Email: \texttt{raowei1998@gmail.com}.}
}
\date{}

\begin{document}
\maketitle

\begin{abstract}
    We prove a dimensional strengthening of the colorful Helly theorem. Let $\mathcal{C}_1,\dots,\mathcal{C}_{d+1}$ be finite nonempty families of convex sets in $\mathbb{R}^d$. If $C_1\cap\cdots\cap C_{d+1}\neq\varnothing$ for every choice of $C_i\in\mathcal{C}_i$, then $\sum_{i=1}^{d+1}\dim(\bigcap\mathcal{C}_i)\geq0$, where $\dim\varnothing=-1$. We also obtain a dimensional strengthening of a theorem of Kim and Lew, in which the intersections are taken over unions of color classes. For simplicial complexes, we introduce the link-Leray dimension, defined as the Leray number of the link for a face and as $-1$ for a nonface. Using this invariant, we prove a matroidal extension of the dimensional colorful Helly theorem for $d$-Leray complexes. We further establish a corresponding strengthening of the topological Kim--Lew theorem in two ranges of parameters. As a special case, we recover the topological colorful Helly theorem of Kalai and Meshulam. Finally, we construct counterexamples to the unrestricted topological extension, even under a stronger local condition.
\end{abstract}



\section{Introduction}\label{sec:introduction}

A family of sets is called \textit{intersecting} if its intersection is not empty. Throughout the paper, we assume that $d$ is a non-negative integer. The Helly theorem~\cite{helly1923mengen} claims that a finite family of convex sets in $\R^d$ is intersecting if any of its subfamilies of size at most $d+1$ is intersecting. Lov\'asz~\cite{barany1982generalization} proved a colorful variation of this result, which can be formulated as follows.

\begin{theorem}[colorful Helly theorem]\label{thm:colorful_helly}
Let $\cC_1,\dots,\cC_{d+1}$ be finite nonempty families of convex sets in $\R^d$. Suppose that for any $C_i\in \mathcal C_i$ for each $i$, we have that $C_1\cap\cdots\cap C_{d+1}\neq\varnothing$, then for at least one $i$, the family $\mathcal C_i$ is intersecting.
\end{theorem}

Arocha et al.~\cite{arocha2009very}, Montejano and Karasev~\cite{montejano2011topological} constructed different generalizations of colorful Helly theorem.

\begin{theorem}[Theorem~10 in~\cite{arocha2009very}]\label{thm:arocha}
    Let $d\geq 1$ and $1\leq k \leq d+1$ be integers. Let $\cC$ be a finite family of convex sets in $\mathbb R^d$, colored with $d+1$ colors. If every subfamily $\cC' \subseteq \cC$ of size $d+1$ having at least $k$ different colors is intersecting, then there are $d+2-k$ color classes whose union is intersecting.
\end{theorem}

\begin{theorem}[Theorem~1.6 in~\cite{montejano2011topological}]\label{thm:montejano}
    Let $d\geq 1$ be an integer, and let $\cC = \bigsqcup_{i=1}^{d+2} \cC_i$ be a finite family of convex sets in $\mathbb R^d$ colored with $d+2$ colors. If for any $\{C_1,\dots,C_{d+2} \}$, where $C_i \in \cC_i$ for $i \in [d+2]$, there is at most one non-intersecting subfamily of size $d+1$, then one of the color classes is intersecting.
\end{theorem}

A broader extension of the colorful Helly theorem, which is a common generalization of Theorem~\ref{thm:arocha} and Theorem~\ref{thm:montejano}, was obtained by Kim and Lew~\cite{Kim_Lew_2024}.

\begin{theorem}[Theorem~1.9 in~\cite{Kim_Lew_2024}]\label{thm:geometric-kim}
Let $d\geq 1$, $r\geq d+1$, $1\leq m\leq r$, and $m\leq k\leq\min\{m+d,r\}$ be integers. Let $\cC$ be a finite family of convex sets in $\R^d$, colored with $r$ different colors, such that $|\cC|\geq\max\{m+d,r\}$. Assume that, for every family $\{A_1,\dots,A_d,B_1,\dots,B_m\}\subseteq\cC$ colored with at least $k$ different colors, at least one of the subfamilies $\{A_1,\dots,A_d,B_i\}$, $i\in[m]$, is intersecting. Then there are $r-k+1$ color classes whose union is intersecting.
\end{theorem}

Throughout the paper, we use the convention $\dim\varnothing=-1$. Our first geometric result is the following dimensional strengthening of the colorful Helly theorem.

\begin{theorem}\label{thm:dim-colorful}
Let $\cC_1,\dots,\cC_{d+1}$ be finite non-empty families of convex sets in $\R^d$. If $C_1\cap\cdots\cap C_{d+1}\neq\varnothing$ for every $C_i\in\cC_i$, then 
\[
\sum_{i=1}^{d+1}\dim\paren{\bigcap\cC_i}\geq 0.
\]
\end{theorem}

In particular, the dimension of the intersection of at least one of the families is at least $0$. Hence, Theorem~\ref{thm:dim-colorful} implies Theorem~\ref{thm:colorful_helly}. We also obtain a similar extension of Theorem~\ref{thm:geometric-kim}. Before stating it, we rephrase the condition in Theorem~\ref{thm:geometric-kim} and clarify some additional cases that arise below.

\begin{definition}\label{def:semi-intersecting}
For an integer $d\geq0$, a finite family $\cG$ of sets of size at least $d+1$ is \emph{$d$-semi-intersecting} if, for every proper subfamily $\cA\subsetneq\cG$ with $\abs{\cA}=d$, there is a set $C\in\cG\setminus\cA$ such that $\bigcap\paren{\cA\cup\{C\}}\neq\varnothing$.
\end{definition}

It is easy to verify that the condition in Theorem~\ref{thm:geometric-kim} is equivalent to requiring that every subfamily of $m+d$ sets, colored with at least $k$ colors, be $d$-semi-intersecting. For convenience, we shall extend Theorem~\ref{thm:geometric-kim} to the cases $d=0$ and $r\leq|\cC|<m+d$. In the latter case, we regard the rephrased condition as asserting that $\cC$ is $d$-semi-intersecting.

For a colored family $\cC=\cC_1\sqcup\cdots\sqcup\cC_r$ and $I\subseteq[r]$, write $K_I(\cC)\coloneqq\bigcap_{i\in I}\paren{\bigcap\cC_i}$.
Let $\binom{[r]}{s}$ be the set of all subsets of $[r]$ of size $s$, where $s$ is a positive integer. The next theorem is our extension of Theorem~\ref{thm:geometric-kim}.

\begin{theorem}\label{thm:geom-main}
Let $d,r,m,k$ be integers such that $d\geq0$, $r\geq d+1$, $1\leq m\leq r$, and $m\leq k\leq\min\{m+d,r\}$. Let $\cC=\cC_1\sqcup\cdots\sqcup\cC_r$ be a finite family of convex sets in $\R^d$, colored with $r$ different colors. Assume that every subfamily of $m+d$ sets, colored with at least $k$ colors, is $d$-semi-intersecting. Set $s=r-k+1$ and
\[
p=\begin{cases}
1,&k=1,\\
d+1,&k\geq2.
\end{cases}
\]
Then there are pairwise distinct sets $I_1,\dots,I_p\in\binomset{[r]}{s}$ such that $\sum_{j=1}^{p}\dim K_{I_j}(\cC)\geq0$.
\end{theorem}

In particular, at least one of the selected intersections is nonempty. Hence, Theorem~\ref{thm:geom-main} implies Theorem~\ref{thm:geometric-kim}. Moreover, by setting $r=d+1$, $m=1$, and $k=d+1$, we recover Theorem~\ref{thm:dim-colorful} from Theorem~\ref{thm:geom-main}.

The paper is organized as follows. Section~\ref{sec:topological-setting} introduces the topological framework and states our results for $d$-Leray complexes and matroids. Section~\ref{sec:prelim} collects auxiliary properties and derives Theorem~\ref{thm:dim-colorful} from Corollary~\ref{cor:top-dimensional-colors}. Section~\ref{sec:geometric} proves the general geometric result, Theorem~\ref{thm:geom-main}, by a minimal-polytope argument and induction on the dimension. Sections~\ref{sec:top-dimensional} and~\ref{sec:top-KL} prove Theorems~\ref{thm:top-dimensional} and~\ref{thm:m=1ork>=m+d-1}, respectively. Finally, Section~\ref{sec:counterexample} presents the counterexample and sharpness constructions in Theorem~\ref{theorem:counterexample-and-sharpness}.

\section{Extensions of the colorful Helly theorem for $d$-Leray complexes}\label{sec:topological-setting}
\subsection{Topological background}

An \textit{abstract simplicial complex}, or simply a \textit{complex}, $K$ is a finite family of sets closed under taking subsets: if $\tau\in K$ and $\sigma\subseteq\tau$, then $\sigma\in K$. The elements of $K$ are called its \textit{faces}, and the union of its faces is its \textit{vertex set}, denoted by $V:=V(K)$. For $W\subseteq V$, the \textit{induced subcomplex} of $K$ on $W$ is $X[W]=\{\sigma\in X:\sigma\subseteq W\}$.
For convenience, we assume that all complexes we consider are distinct from the void complex $\varnothing$; it should not be confused with the \textit{empty complex} $\{\varnothing\}$.

For any $\tau\in X$, its \textit{link} in $K$ is defined as follows
\[
\lk(K,\tau)
\coloneqq\{\sigma\in X:\sigma\cap \tau=\varnothing
\text{ and }\sigma\cup \tau\in K\}.
\]

For an abstract simplicial complex $K$ and $i\geq -1$, we denote its \textit{$i$th reduced homology group over $\mathbb Z$} by $\widetilde{H}_i(K)$. 
In particular, we use the convention that $\widetilde H_{-1}(\{\varnothing\})\cong\mathbb Z$, whereas $\widetilde H_{-1}(X)=0$ when $X$ has a vertex. For a detailed introduction of this notion, we refer the reader to~\cite{hatcher2002algebraic}.

For a simplicial complex $K$ with vertex set $V$, its \emph{Leray number} is defined as follows
\[
L(K)=\min \big\{j\geq 0: \widetilde H_i(K[W])=0\text{ for every $W\subseteq V$ and every $i\geq j$}\big\}.
\]
We say that $K$ is a \textit{$d$-Leray complex} if $L(K)\leq d$. By definition, $L(K)\geq L(K[W])$ for any $W\subseteq V$.

Also, we use the following property of the
Leray number
\begin{equation}
\label{equation characterization of Leray}
L(K)=\max_{\sigma\in K}\ell_K(\sigma),
\end{equation}
where 
\begin{equation}\label{equation definition of ell}
\ell_K(\sigma)
\coloneqq \min\big\{j\geq -1 \colon
\widetilde H_i(\lk(K,\sigma))=0
\text{ for every }i\geq j\big\}.
\end{equation}
For details, we refer the reader to Proposition~3.1 in~\cite{kalai2006intersections}.

Next, we introduce an invariant, which plays in our topological results a role of dimension in Theorem~\ref{thm:geom-main}.

\begin{definition}\label{def:link-Leray-dimension}
Let $K$ be a simplicial complex with vertex set $V$. For every $W\subseteq V$, its \emph{link-Leray dimension} with respect to $K$ is defined as follows 
\[
\delta_K(W)\coloneqq 
\begin{cases}L(\lk(K,W)) &\text{if }W\in K;\\
-1, &\text{otherwise.}
\end{cases}
\]
\end{definition}

A \textit{matroid} $M$ is an abstract simplicial complex containing $\varnothing$ and satisfying the exchange axiom: if $A,B\in M$ and $|A|<|B|$, then there exists $x\in B\setminus A$ such that $A\cup\{x\}\in M$. The members of $M$ are called its \textit{independent sets}. An independent set maximal under inclusion is called a \textit{basis} of $M$. It is well known that all bases of $M$ have the same cardinality, called the \textit{rank} of $M$.

Let $V$ be the vertex set of the matroid $M$. It is well-known that for every $W\subset V$, the complex $M[W]$ is a matroid. The \textit{rank function} is the map $\rho_M:2^V\to\N$ such that for every $W\subseteq V$, the value $\rho_M(W)$ is the rank of $M[W]$.

The \textit{span} of a set $A\subseteq V$ is defined as follows
\[
\Span_M(A)\coloneqq \{v\in V:\rho_M(A\cup\{v\})=\rho_M(A)\}.
\]

We abbreviate $\rho_M$ and $\Span_M$ to $\rho$ and $\Span$, respectively, whenever the matroid is unambiguous. We refer to~\cite{oxley2006matroid} for more background.

\subsection{Previous topological extensions}

Recall a result of Kalai and Meshulam extending the colorful Helly theorem; see Theorem 1.6 in~\cite{kalai2005topological}.

\begin{theorem}\label{thm:KM-intro}
Let $X$ be a $d$-Leray simplicial complex. Let $M$ be a matroid such that it is a subcomplex of $X$ with the same vertex set. Then there is a face $\tau\in X$ such that $\rho_M(V\setminus\tau)\leq d$.
\end{theorem}

Very recently, Kim and Lew proved the following extension of this theorem; see Theorem 1.7 in~\cite{Kim_Lew_2024}.

\begin{theorem}
\label{thm:topo_kim}
Let $r, m,$ and $k$ be integers such that $d\geq1$, $r\geq d+1$, $1\leq m\leq r$, and $m\leq k\leq\min\{m+d,r\}$. Let $V$ be a finite set with $|V|\geq\max\{m+d,r\}$. Let $X$ be a $d$-Leray simplicial complex with vertex set $V$. Let $M$ be a matroid of rank $r$ with vertex set $V$.

Assume that for every $U=\{u_1,\dots,u_{d+1},v_1,\dots,v_{m-1}\}\subseteq V$ with $\rho(U)\geq k$, either $\{u_1,\dots,u_{d+1}\}\in X$ or there exists $j\in[m-1]$ such that $\sigma\cup\{v_j\}\in X$ for every $\sigma\subsetneq\{u_1,\dots,u_{d+1}\}$. Then there is a face $\tau\in X$ such that $\rho(V\setminus\tau)\leq k-1$.
\end{theorem}

In addition to these closely related results, there have been a recent development~\cite{soberon2026kkm}.

\subsection{Our contribution}\label{sec:our-contribution}

Our first topological result uses the link-Leray dimension introduced above.

\begin{theorem}\label{thm:top-dimensional}
Let $r$ be an integer with $r\geq d+1$. Let $X$ be a $d$-Leray simplicial complex. Let $M$ be a matroid of rank $r$ such that it is a subcomplex of $X$ with the same vertex set. 

Then $M$ has a basis $B$ such that
\[
  \sum_{v\in B}
  \delta_X\paren{V\setminus\Span(B\setminus\{v\})}
  \geq0.
\]
\end{theorem}

\begin{corollary}\label{cor:top-dimensional-colors}
Let $r$ be an integer with $r\geq d+1$. Let $X$ be a $d$-Leray complex on $V=V_1\sqcup\cdots\sqcup V_r$, with $M=V_1 * \dots * V_r \subseteq X$. Then we have $\sum_{i=1}^r\delta_X(V_i)\geq0$.
\end{corollary}

We shall discuss later how to obtain Theorem~\ref{thm:dim-colorful} from Corollary~\ref{cor:top-dimensional-colors}. Furthermore, motivated by Theorems~\ref{thm:topo_kim} and~\ref{thm:geom-main}, we considered the following problem.

\begin{conjecture}\label{conjecture: generalization of Kim-Lew's result}
Let $d\geq 1$, $r\geq d+1$, $1\leq m\leq r$, and $m\leq k\leq\min\{m+d,r\}$ be integers. Let $V$ be a finite set with $|V|\geq\max\{m+d,r\}$. Let $X$ be a $d$-Leray simplicial complex on vertex set $V$, and let $M$ be a rank-$r$ matroid on vertex set $V$ with rank function $\rho$.

Assume that, for every $U=\{u_1,\dots,u_{d+1},v_1,\dots,v_{m-1}\}\subseteq V$ with $\rho(U)\geq k$, either $\{u_1,\dots,u_{d+1}\}\in X$ or there exists $j\in[m-1]$ such that $\sigma\cup\{v_j\}\in X$ for every $\sigma\subsetneq\{u_1,\dots,u_{d+1}\}$. Set
\[
p=\begin{cases}
1,&k=1,\\
d+1,&k\geq2.
\end{cases}
\]
Then there exist a basis $B$ of $M$ and $p$ distinct subsets $I_1,\dots,I_p\subseteq B$ with $|I_i|=r-k+1$ for all $i\in[p]$ such that
\[
\sum_{i=1}^{p}\delta_X(V\setminus\Span(B\setminus I_i))\geq0.
\]
\end{conjecture}

However, we found that Conjecture~\ref{conjecture: generalization of Kim-Lew's result} is false in general. We shall show that it holds in the following special cases.

\begin{theorem}\label{thm:m=1ork>=m+d-1}
Conjecture~\ref{conjecture: generalization of Kim-Lew's result} holds if $m=1$ or $k\geq m+d-1$.
\end{theorem}

Note that when $m=1$ and $k=d+1$, Theorem~\ref{thm:m=1ork>=m+d-1} implies Theorem~\ref{thm:KM-intro}. Indeed, if $r:=\rho_M(V)\le d$, we may take $\tau=\varnothing$. Otherwise, the inclusion $M\subseteq X$ ensures that every $(d+1)$-set of rank $d+1$ is a face of $X$, so the hypothesis holds with $m=1$ and $k=d+1$. The nonnegative sum in the conclusion yields a set $I\subseteq B$ of size $r-d$ such that $\delta_X(V\setminus\Span_M(B\setminus I))\ge0$. Thus, $\tau:=V\setminus\Span_M(B\setminus I)$ is a face of $X$, and $\rho_M(V\setminus\tau)=\rho_M(\Span_M(B\setminus I))=|B\setminus I|=d$, as required.

We also give an example showing that, even under a stronger condition, Conjecture~\ref{conjecture: generalization of Kim-Lew's result} is false in general. Following Remark~1.8 in~\cite{Kim_Lew_2024}, we say that a pair $(X,M)$ on a common vertex set $V$, where $\rho$ is the rank function of $M$, satisfies the \emph{strengthened Kim--Lew condition} for the parameters $(d,m,k)$ if, for every $U\subseteq V$ with $|U|=d+m$ and $\rho(U)\geq k$,
\begin{equation}\label{eq:strong-KL-condition}
    \bigl|\{A\in \tbinom{U}{d+1}:A\notin X\}\bigr|\leq m-1.
\end{equation}

The following theorem gives a counterexample to Conjecture~\ref{conjecture: generalization of Kim-Lew's result} and, at the same time, determines the optimal number of subsets in the two positive cases covered by Theorem~\ref{thm:m=1ork>=m+d-1}.

\begin{theorem}\label{theorem:counterexample-and-sharpness}
The following statements hold.
\begin{enumerate}[label=\textup{(\roman*)}]
    \item For every $d\geq 2$, there exist a $d$-Leray simplicial complex $X$ and a rank-$r$ matroid $M$, where $r=d+1$, such that the strengthened Kim--Lew condition holds for $m=k=2$, but, for every basis $B$ of $M$,
    \begin{equation}\label{eq:counterexample-total}
        \sum_{I\in\binom{B}{d}}
        \delta_X\bigl(V\setminus\Span(B\setminus I)\bigr)=-1.
    \end{equation}
    In particular, the conclusion of Conjecture~\ref{conjecture: generalization of Kim-Lew's result} fails, even under the strengthened condition~\eqref{eq:strong-KL-condition}.

    \item Let $d\geq 1$, $r\geq d+1$, and $2\leq k\leq r$, and set $s=r-k+1$. There exist a $d$-Leray simplicial complex $X$ and a rank-$r$ matroid $M$ such that, for every basis $B$ of $M$ and every $1\leq \ell\leq \binom{r}{s}$,
    \begin{equation}\label{eq:exact-sharpness-profile}
        \max_{\substack{
            \mathcal I\subseteq\binom{B}{s}\\
            |\mathcal I|=\ell}}
        \ \sum_{I\in\mathcal I}
        \delta_X\bigl(V\setminus\Span(B\setminus I)\bigr)
        =d+1-\ell.
    \end{equation}
    Moreover, the construction may be chosen so that:
    \begin{enumerate}[label=\textup{(\alph*)}]
        \item if $2\leq k\leq d+1$, it satisfies the hypothesis of Theorem~\ref{thm:m=1ork>=m+d-1} with $m=1$;
        \item if $k\geq m+d-1$, it satisfies the hypothesis of Theorem~\ref{thm:m=1ork>=m+d-1} with $k\geq m+d-1$.
    \end{enumerate}
\end{enumerate}
Consequently, whenever at least $d+2$ pairwise distinct subsets of the required size are available, $d+1$ is the largest number for which a nonnegative sum can be guaranteed in either partial theorem: the maximum in \eqref{eq:exact-sharpness-profile} is $0$ for $\ell=d+1$ and is $-1$ for $\ell=d+2$.
\end{theorem}

Thus, the picture is complete in the following sense. The proposed extension is false in the full range of parameters, even under the stronger condition from Remark~1.8 in~\cite{Kim_Lew_2024}; on the other hand, in each of the two ranges where the extension is valid, the number $d+1$ in the conclusion is optimal.

\section{Preliminaries and auxiliary properties}\label{sec:prelim}

This section collects the standard facts, technical tools used in the proofs, and the proof of Theorem~\ref{thm:dim-colorful} using Corollary~\ref{cor:top-dimensional-colors}.

\subsection{Simplicial complexes and Leray properties}

We first record two elementary set-theoretic identities for links.

\begin{lemma}\label{lem:link-identities}
Let $X$ be an abstract simplicial complex with vertex set $V$.
If $\tau\in X$ and $\sigma \subseteq \tau$, then $\lk(X,\tau)=\lk(\lk(X,\sigma),\tau\setminus\sigma)$.
\end{lemma}

\begin{proof}
Using the definition of the link and $\sigma \subset \tau\in X$, we get the following sequence of equivalences
\begin{align*}
    \eta\in \lk (\lk(X,\sigma), \tau\setminus \sigma)&\Longleftrightarrow \eta\cap (\tau\setminus \sigma)=\varnothing, \eta \cup (\tau\setminus \sigma)\in \lk(X, \sigma)\\
    &\Longleftrightarrow \eta\cap(\tau\setminus\sigma)=\varnothing, (\eta\cup (\tau\setminus \sigma))\cap \sigma=\varnothing, (\eta\cup (\tau\setminus \sigma))\cup \sigma\in X\\
    &\Longleftrightarrow \eta \cap \tau=\varnothing,  \eta\cup \tau \in X\\
    &\Longleftrightarrow \eta \in \lk(X, \tau),
\end{align*}
which finishes the proof.
\end{proof}

\begin{lemma}\label{lem:delta-monotone}
Let $X$ be an abstract simplicial complex. If $\sigma\subseteq\tau\in X$, then $\delta_X(\sigma)\geq\delta_X(\tau)\geq\ell_X(\tau)$.
\end{lemma}

\begin{proof}
Set
\[
Y=\lk(X,\sigma),\qquad
\eta=\tau\setminus\sigma,\qquad
Z=\lk(X,\tau).
\]
By Lemma~\ref{lem:link-identities}, we have
\[
Z=\lk(X,\tau)=\lk(\lk(X,\sigma),\tau\setminus \sigma)=\lk(Y,\eta).
\]
Hence, for every $\rho\in Z$, we have $\eta\cap \rho=\varnothing$ and $\eta\cup\rho\in Y$. Applying
Lemma~\ref{lem:link-identities} to
$\eta\subseteq\eta\cup\rho\in Y$ gives
\[
\lk(Y,\eta\cup\rho)
=\lk\bigl(\lk(Y,\eta),\rho\bigr)
=\lk(Z,\rho).
\]
It follows from~\eqref{equation definition of ell} that
\[
\ell_Y(\eta\cup\rho)=\ell_Z(\rho).
\]
Therefore, using~\eqref{equation characterization of Leray}, we obtain
\begin{align*}
\delta_X(\sigma)
&=L(Y)
 =\max_{\gamma\in Y}\ell_Y(\gamma)\\
&\geq \max_{\rho\in Z}\ell_Y(\eta\cup\rho)
 =\max_{\rho\in Z}\ell_Z(\rho)\\
&=L(Z)
 =\delta_X(\tau).
\end{align*}
Finally, since $\varnothing\in Z$, we have
\[
\delta_X(\tau)
=L(Z)
=\max_{\rho\in Z}\ell_Z(\rho)
\geq \ell_Z(\varnothing)
=\ell_X(\tau),
\]
which finishes the proof.
\end{proof}

We use the following Helly-type theorem for $d$-Leray complexes; see Lemma~5.3 in~\cite{Kim_Lew_2024}.

\begin{lemma}\label{lem:Leray-Helly}
Let $X$ be a $d$-Leray complex with vertex set $V$, and let $U\subseteq V$. If any subset of $U$ of size at most $d+1$ is a face of $X$, then $U$ is also its face.
\end{lemma}

We also use the following boundary lemma due to Kim and Lew; see Lemma 5.4 in~\cite{Kim_Lew_2024}.

\begin{lemma}\label{lem:KL-boundary}
Let $X$ be a $d$-Leray complex on $V$, and let $A\subseteq V$ have size $d+1$ and satisfy $A\notin X$. Put
\[
U=\set{v\in V\setminus A:
\sigma\cup\{v\}\in X
\text{ for every }\sigma\in\tbinom{A}{d}}.
\]
If $U\neq\varnothing$, then $U\cup(A\setminus\{a\})\in X$ for every $a\in A$.
\end{lemma}

\subsection{Proof of Theorem~\ref{thm:dim-colorful}}
Recall that the \textit{nerve} of a finite family $\cF$ is the abstract simplicial complex $N(\cF)=\{\cA\subseteq\cF:\bigcap\cA\neq\varnothing\}$. It is known that the nerve of a finite family of convex sets in $\R^d$ is $d$-Leray; see, for example, \cite{wegner1975d} and Section~2 in~\cite{tancer2012intersection}. Then it follows the following observation for the dimension and Definition~\ref{def:link-Leray-dimension}.

\begin{observation}\label{ob:dim-delta}
    Let $\cF$ be a finite family of convex sets in $\mathbb R^d$. For any $\cF' \subseteq \cF$, we have $\dim \left(\bigcap \cF' \right) \geq \delta_{N(\cF)}(\cF')$.
\end{observation}

\begin{proof}
    If $\bigcap\cF'=\varnothing$, then $\cF'\notin N(\cF)$, and hence $\delta_{N(\cF)}(\cF')=-1=\dim\paren{\bigcap\cF'}$. Thus, we may assume that $\bigcap\cF'\neq\varnothing$. Put $t=\dim\paren{\bigcap\cF'}$.

By the definitions of the nerve and the link, for every $\cA\subseteq\cF\setminus\cF'$, we have
\[
\cA\in\lk(N(\cF),\cF')
\quad\Longleftrightarrow\quad
\cA\cup\cF'\in N(\cF)
\quad\Longleftrightarrow\quad
\paren{\bigcap\cF'}\cap\paren{\bigcap\cA}\neq\varnothing.
\]
Consequently, $\lk(N(\cF),\cF')$ is the nerve of the family $\bigl\{ F\cap\bigcap\cF':F\in\cF\setminus\cF' \bigr\}$, after discarding its empty members. All members of this family are convex and lie in the affine hull of $\bigcap\cF'$, which can be identified with $\R^t$. Hence, $\lk(N(\cF),\cF')$ is $t$-Leray. By Definition~\ref{def:link-Leray-dimension}, we obtain
\[
\delta_{N(\cF)}(\cF')
=L\bigl(\lk(N(\cF),\cF')\bigr)
\leq t
=\dim\paren{\bigcap\cF'},
\]
which finishes the proof.
\end{proof}

Then, we can give the proof of Theorem~\ref{thm:dim-colorful} using Corollary~\ref{cor:top-dimensional-colors}.

\begin{proof}[Proof of Theorem~\ref{thm:dim-colorful}]
Let $\cC=\cC_1\sqcup\cdots\sqcup\cC_{d+1}$. By the assumption, we have $\cC_1*\cdots*\cC_{d+1}\subseteq N(\cC)$. Since $N(\cC)$ is $d$-Leray, we may apply Corollary~\ref{cor:top-dimensional-colors} with $r=d+1$. Together with Observation~\ref{ob:dim-delta}, this gives
\[
\sum_{i=1}^{d+1}\dim\paren{\bigcap\cC_i}
\geq \sum_{i=1}^{d+1}\delta_{N(\cC)}(\cC_i)
\geq 0,
\]
which finishes the proof.
\end{proof}

\subsection{Matroids and tolerance complexes}

We first recall two standard matroid constructions. Let $a\in\N$ and let $E$ be a set. The uniform matroid $U_{a,E}$ has ground set $E$ and contains all subsets of $E$ of size at most $a$. Given matroids $M_1$ and $M_2$, the direct sum $M_1\oplus M_2$ is the matroid whose ground set is the disjoint union of $V(M_1)$ and $V(M_2)$ and whose independent sets are exactly the unions $I_1\cup I_2$ with $I_1\in M_1$ and $I_2\in M_2$.

Let $M$ be a matroid on $V$ with rank function $\rho$. We shall repeatedly use the following elementary rank inequality: for any $U\subseteq W\subseteq V$,
\[
\rho(U)\geq\rho(W)-\abs{W\setminus U}.
\]

Tolerance complexes of matroids were studied by Kim and Lew~\cite{kim2023leray, Kim_Lew_2024}.

\begin{definition}\label{def:tolerance-matroid}
Let $M$ be a rank-$r$ matroid on $V$ with rank function $\rho$, and let $0\leq t\leq r$. The $t$-tolerance complex of $M$ is
\[
M^t=\{\sigma\subseteq V:\rho(\sigma)\geq\abs\sigma-t\}.
\]
In particular, $M^0=M$.
\end{definition}

The following proposition records the rank function of a tolerance complex.

\begin{proposition}[Lemmas~3.1 and~3.2 in~\cite{Kim_Lew_2024}]\label{prop:tolerance-rank}
The complex $M^t$ is a matroid. If $\rho^t$ denotes its rank function, then for every $A\subseteq V$,
\[
\rho^t(A)=\min\{\abs A,\rho(A)+t\}.
\]
\end{proposition}

The next obstruction lemma is implicit in the proof of Theorem~\ref{thm:KM-intro} and is stated explicitly as Lemma~5.2 in~\cite{Kim_Lew_2024}.

\begin{lemma}\label{lem:KM-obstruction}
Let $X$ be a simplicial complex and let $M$ be a matroid on the same vertex set $V$. If $\rho_M(V\setminus\sigma)\geq\ell_X(\sigma)+1$ for every $\sigma\in X$, then $M\nsubseteq X$.
\end{lemma}

\section{Proof of Theorem~\ref{thm:geom-main}}\label{sec:geometric}
\subsection{Auxiliary lemmas}
We begin with a reduction and several lemmas. Since $\cC$ is finite, a standard argument allows us to assume that all sets are polytopes. The first lemma shows that one constraint can be removed when at most $d+1$ convex sets cut a polytope down to a vertex.

\begin{lemma}\label{lem:remove-constraint}
Let $Q\subseteq\R^d$ be a polytope, let $v$ be a vertex of $Q$, and let $D_1,\dots,D_{d+1}$ be convex sets. If $Q\cap D_1\cap\cdots\cap D_{d+1}=\{v\}$, then some subfamily of at most $d$ of the sets $D_i$ has intersection with $Q$ equal to $\{v\}$.
\end{lemma}

\begin{proof}
Suppose, to the contrary, that for every $i\in[d+1]$ there is a point
\[
q_i\in Q\cap\bigcap_{j\neq i}D_j\setminus\{v\}.
\]
Since $v$ is a vertex of $Q$, there is an affine hyperplane $H$ that strictly separates $v$ from the finite set $\{q_1,\dots,q_{d+1}\}$. For each $i$, the segment $[v,q_i]$ is contained in $Q\cap\bigcap_{j\neq i}D_j$ and meets $H$. Hence, every $d$ of the $d+1$ convex sets $Q\cap H\cap D_1,\dots,Q\cap H\cap D_{d+1}$ intersect. These sets lie in an affine space of dimension at most $d-1$. By Helly's theorem, all of them intersect. This contradicts $Q\cap D_1\cap\cdots\cap D_{d+1}=\{v\}$ because $v\notin H$.
\end{proof}

If $\cA\subseteq\cC=\cC_1\sqcup\cdots\sqcup\cC_r$, then its set of colors is $\col(\cA)=\{i\in[r]:\cA\cap\cC_i\neq\varnothing\}$. Thus, $\cA$ uses $\abs{\col(\cA)}$ colors. For convenience, we define a \textit{required pair} $(\cG,\cA)$ as follows:
\[
\begin{cases}
\abs{\cG}=d+m,\quad \abs{\col(\cG)}\geq k,\quad
\cA\subsetneq\cG,\quad \abs{\cA}=d,
&\text{if }\abs{\cC}\geq d+m,\\[2mm]
\cG=\cC,\quad \cA\subsetneq\cC,\quad \abs{\cA}=d,
&\text{if }r\leq\abs{\cC}<d+m.
\end{cases}
\]
Thus, the condition in Theorem~\ref{thm:geom-main} can be rephrased as follows: for every required pair $(\cG,\cA)$, there is a member $C\in\cG\setminus\cA$ such that $\bigcap\bigl(\cA\cup\{C\}\bigr)\neq\varnothing$.

For a nonempty polytope $P$ and $i\in[r]$, write $\cC_i[P]=\{C\cap P:C\in\cC_i\}$, with the original indexing retained. We say that $P$ \emph{retains} the condition in Theorem~\ref{thm:geom-main} if the colored family $\cC[P]=\bigsqcup_{i=1}^r\cC_i[P]$ satisfies that condition. Equivalently, for every required pair $(\cG,\cA)$, there is a member $C\in\cG\setminus\cA$ such that
\[
P\cap\bigcap\bigl(\cA\cup\{C\}\bigr)\neq\varnothing.
\]

The second lemma gives a convenient form of the minimal-polytope argument.

\begin{lemma}\label{lem:minimal-retaining}
There is an inclusion-minimal nonempty polytope $Q$ that retains the condition. Moreover, for every vertex $v$ of $Q$, there is a required pair $(\cG,\cA)$ such that
\begin{equation}\label{eq:vertex-certificate-old}
Q\cap\bigcap\bigl(\cA\cup\{C\}\bigr)\subseteq\{v\}
\qquad\text{for every }C\in\cG\setminus\cA.
\end{equation}
\end{lemma}

\begin{proof}
Let $\mathscr P$ be the set of all nonempty retaining polytopes, ordered by inclusion. Let $\mathscr P'\subseteq\mathscr P$ be a chain, and put $L=\bigcap_{P\in\mathscr P'}P$. The members of the chain are nonempty compact sets, and every finite subcollection has nonempty intersection. Hence, $L$ is nonempty, compact, and convex.

Fix a required pair $(\cG,\cA)$. We claim that there is a $C_{\cG,\cA}\in\cG\setminus\cA$ such that, for every $P\in\mathscr P'$,
\[
P\cap\bigcap\bigl(\cA\cup\{C_{\cG,\cA}\}\bigr)\neq\varnothing.
\]
Otherwise, for every $C\in\cG\setminus\cA$, choose a member $P_C\in\mathscr P'$ disjoint from $\bigcap(\cA\cup\{C\})$. Since only finitely many sets $P_C$ occur and they belong to a chain, one of them is contained in all the others. It would be disjoint from $\bigcap(\cA\cup\{C\})$ for every $C\in\cG\setminus\cA$, contradicting its retaining property. The compact sets
\[
\biggl\{P\cap\bigcap\bigl(\cA\cup\{C_{\cG,\cA}\}\bigr)\biggr\}_{P\in\mathscr P'}
\]
are therefore nonempty and nested. Moreover, since every finite subcollection has nonempty intersection, we have
\[
\bigcap_{P\in\mathscr P'}\biggl(P\cap\bigcap\bigl(\cA\cup\{C_{\cG,\cA}\}\bigr)\biggr)
=L\cap\bigcap\bigl(\cA\cup\{C_{\cG,\cA}\}\bigr)
\neq\varnothing.
\]

Choose one such intersection point for every required pair, and let $P_0$ be the convex hull of the finitely many chosen points. Then $P_0\subseteq L$ is a nonempty retaining polytope. Thus, every chain has a lower bound in $\mathscr P$, and the dual form of Zorn's lemma gives an inclusion-minimal retaining polytope $Q$.

Let $v$ be a vertex of $Q$. Suppose that \eqref{eq:vertex-certificate-old} fails for every required pair. For each required pair $(\cG,\cA)$, choose $C_{\cG,\cA}\in\cG\setminus\cA$ and
\[
q_{\cG,\cA}\in
Q\cap\bigcap\bigl(\cA\cup\{C_{\cG,\cA}\}\bigr)
\setminus\{v\}.
\]
The convex hull of all these points is retaining and is contained in $Q$. It does not contain the extreme point $v$, and hence it is a proper retaining subpolytope of $Q$, contradicting minimality.
\end{proof}

\begin{lemma}\label{lem:good-colors}
Let $Q$ be an inclusion-minimal retaining polytope. Then every vertex $v$ of $Q$ belongs to every member of at least $r-k+1$ color classes.
\end{lemma}

\begin{proof}
Call a color \emph{bad at $v$} if it has a member that does not contain $v$. We show that fewer than $k$ colors are bad.

By Lemma~\ref{lem:minimal-retaining}, there is a required pair $(\cG,\cA)$ satisfying \eqref{eq:vertex-certificate-old}. Since $Q$ retains this requirement, there is a member $C_0\in\cG\setminus\cA$ such that
\[
Q\cap\bigcap\bigl(\cA\cup\{C_0\}\bigr)=\{v\}.
\]
The family $\cA\cup\{C_0\}$ has $d+1$ members. Applying Lemma~\ref{lem:remove-constraint}, we obtain a subfamily $\cA_v\subseteq\cA\cup\{C_0\}$ with $|\cA_v|=d$ such that $Q\cap\bigcap\cA_v=\{v\}$. In particular, every member of $\cA_v$ contains $v$.

Suppose first that $\abs{\cC}\geq d+m$. Since $\cG\setminus\cA_v$ has $m$ members, we have $\abs{\col(\cA_v)}\geq k-m$. Assume, for a contradiction, that at least $k$ colors are bad at $v$. Since $\abs{\col(\cA_v)}\geq k-m$, we can choose $m$ distinct bad members forming a family $\cB$, disjoint from $\cA_v$, such that $\abs{\col(\cA_v\cup\cB)}\geq k$.

Indeed, first choose bad members from enough colors outside $\col(\cA_v)$ to reach $k$ colors, and then complete the family to $m$ bad members; this is possible because $k\geq m$. The family $\cA_v\cup\cB$ has exactly $d+m$ members, so the condition gives a $C\in\cB$ such that
\[
Q\cap\bigcap\bigl(\cA_v\cup\{C\}\bigr)\neq\varnothing.
\]
However, we have $Q\cap\bigcap\cA_v=\{v\}$, while $v \notin C$ by the choice of $\mathcal B$, a contradiction.

Now suppose that $r\leq\abs{\cC}<d+m$. In this case, the whole family $\cC[Q]$ is $d$-semi-intersecting. Applying the $d$-semi-intersection property once more gives a member of $\cC\setminus\cA_v$ that contains $v$. All bad members lie outside $\cA_v$, and at least one further member outside $\cA_v$ is good. Hence, the number of bad members is at most
\[
\abs{\cC}-d-1\leq m-2<k.
\]
Thus, fewer than $k$ colors are bad in both cases. Therefore, at least $r-k+1$ colors are good at $v$.
\end{proof}

\subsection{Proof of Theorem~\ref{thm:geom-main}}
We first show the existence of one $I\in\binom{[r]}{r-k+1}$ such that $K_I(\cC)\neq\varnothing$.

\begin{theorem}\label{thm:point-KL}
Under the condition of Theorem~\ref{thm:geom-main}, there is an $I\in\binomset{[r]}{r-k+1}$ such that $K_I(\cC)\neq\varnothing$.
\end{theorem}

\begin{proof}
Choose an inclusion-minimal retaining polytope $Q$ by Lemma~\ref{lem:minimal-retaining}, and let $v$ be a vertex of $Q$. By Lemma~\ref{lem:good-colors}, at least $r-k+1$ color classes have the property that every one of their members contains $v$. Choosing any $r-k+1$ of these colors gives the desired set $I$.
\end{proof}

The next lemma is the dimension-reduction step used in the induction. We say that a parameter tuple $(d,r,m,k)$ is \textit{admissible} if it satisfies the conditions in Theorem~\ref{thm:geom-main}.

\begin{lemma}\label{lem:dimension-reduction}
Let $Q$ be a polytope that retains the condition for the parameter tuple $(d,r,m,k)$, and put $d'=\dim Q<d$. Let $R\subseteq[r]$, put $j=r-\abs R$, and suppose that $0\leq j\leq d-d'$ and $j\leq k-1$. Set $r'=r-j$, $k'=k-j$, and $m'=\min\{m,k'\}$, and assume that $k'\leq m'+d'$. Then $(d',r',m',k')$ is admissible, and the colored family $\cD_R=\bigsqcup_{i\in R}\cC_i[Q]$ in $\operatorname{aff}Q\cong\R^{d'}$ satisfies the condition for these reduced parameters.
\end{lemma}

\begin{proof}
The inequalities $j\leq d-d'$ and $j\leq k-1$ give $r'=r-j\geq d'+1$ and $1\leq k'\leq r'$. We also have $1\leq m'\leq k'$. Hence, the remaining admissibility inequality is the assumed bound $k'\leq m'+d'$.

For the second part of the lemma, we first consider the case $\abs{\cD_R}\geq d'+m'$. Let $\cG\subseteq\cD_R$ with $\abs{\cG}=d'+m'$ and $\abs{\col(\cG)}\geq k'$. Suppose that $\cG$ is not $d'$-semi-intersecting. Then there is a proper subfamily $\cA\subsetneq\cG$ with $\abs{\cA}=d'$ such that, for every $B\in\cG\setminus\cA$, we have $\bigcap\bigl(\cA\cup\{B\}\bigr)=\varnothing$. Choose an arbitrary $B_0\in\cG\setminus\cA$, and put $\mathcal E=\cA\cup\{B_0\}$. Then
\begin{equation}\label{eq:empty-small-witness}
\bigcap\mathcal E=\varnothing,
\qquad
\abs{\mathcal E}=d'+1\leq d.
\end{equation}
If $r\leq\abs{\cC}<d+m$, then the whole family $\cC[Q]$ contains $\mathcal E$. By \eqref{eq:empty-small-witness}, it is not $d$-semi-intersecting, contradicting the condition.

Assume, therefore, that $\abs{\cC}\geq d+m$. Add to $\cG$ one indexed member from each of the $j$ omitted color classes. The enlarged family uses at least $k'+j=k$ colors and has $d'+m'+j\leq d+m$ members. Complete it arbitrarily to a $(d+m)$-member subfamily $\cF\subseteq\cC[Q]$. The family $\cF$ contains the subfamily $\mathcal E$, whose intersection is empty and which has at most $d$ members, so $\cF$ is not $d$-semi-intersecting. This contradicts the condition.

It remains to consider the case $r'\leq\abs{\cD_R}<d'+m'$. Suppose that the whole family $\cD_R$ is not $d'$-semi-intersecting. As above, it contains a subfamily $\mathcal E$ satisfying \eqref{eq:empty-small-witness}. If $\abs{\cC}<d+m$, this again contradicts the condition.

If $\abs{\cC}\geq d+m$, add one member from every omitted color class to the whole family $\cD_R$. The resulting family uses all $r$ colors and has $\abs{\cD_R}+j<d'+m'+j\leq d+m$ members. Complete it to a $(d+m)$-member subfamily of $\cC[Q]$. The completed family contains $\mathcal E$ and hence is not $d$-semi-intersecting, again a contradiction. This completes the proof.
\end{proof}

\begin{proof}[Proof of Theorem~\ref{thm:geom-main}]
If $k=1$, then $m=1$, $s=r$, and $\binomset{[r]}{s}= \{[r] \}$. Theorem~\ref{thm:point-KL} gives $K_{[r]}(\cC)\neq\varnothing$, which proves the assertion. We may therefore assume that $k\geq2$. We argue by induction on $d$. When $d=0$, the assertion is trivial. Assume that $d\geq1$ and that the result holds in all smaller dimensions.

Choose an inclusion-minimal retaining polytope $Q$, and put $d'=\dim Q$, $h=d-d'$, and $s=r-k+1$. Suppose first that $d'=d$. Choose $d+1$ affinely independent vertices of $Q$. By Lemma~\ref{lem:good-colors}, assign to each of these vertices an $s$-subset of colors all of whose members contain that vertex.

Let the distinct assigned subsets be $I_1,\dots,I_c$, and let $t_a$ be the number of selected vertices assigned to $I_a$. The corresponding $t_a$ vertices are affinely independent and belong to $K_{I_a}(\cC)$, so $\dim K_{I_a}(\cC)\geq t_a-1$. Since $k\geq2$, we have $1\leq s\leq r-1$, and therefore $\binom{r}{s}\geq r\geq d+1$. Choose a further $d+1-c$ distinct $s$-subsets, different from $I_1,\dots,I_c$. Their intersections have dimension at least $-1$. Hence, the sum of the dimensions of all $d+1$ selected intersections is at least
\[
\sum_{a=1}^c(t_a-1)-(d+1-c)
=(d+1-c)-(d+1-c)=0.
\]

Assume now that $d'<d$, so $h\geq1$. Define $r'=\max\{r-h,s\}$, $k'=\max\{1,k-h\}$, and $m'=\min\{m,k'\}$. Let $j=r-r'$. Then we have $j=\min\{h,k-1\}$, $k'=k-j$, and $r'-k'+1=r-k+1=s$. We distinguish two cases.

\smallskip
\noindent\emph{Case 1: $h\geq k-1$.}
Then $j=k-1$, $r'=s$, and $k'=m'=1$. Moreover, we have
\[
s=r-k+1\geq d-k+2\geq d'+1.
\]
Hence, the reduced parameter tuple $(d',r',m',k')$ is admissible. For every $R\in\binomset{[r]}{s}$, Lemma~\ref{lem:dimension-reduction} and Theorem~\ref{thm:point-KL}, applied in $\operatorname{aff}Q$, give $Q\cap K_R(\cC)\neq\varnothing$. Thus, every $s$-subset of $[r]$ gives a nonempty intersection. Since $\binom{r}{s}\geq d+1$, any $d+1$ distinct choices prove the theorem.

\smallskip
\noindent\emph{Case 2: $h\leq k-2$.}
Now $j=h$, $r'=r-h$, $k'=k-h\geq2$, and $m'=\min\{m,k'\}$. The reduced parameters are admissible. Indeed, $r'=r-h=r-(d-d')\geq d'+1$. If $m'=k'$, then $k'\leq m'+d'$ is immediate; otherwise, $m'=m$, and
\[
k'=k-h\leq m+d-h=m'+d'.
\]

For every $R\in\binomset{[r]}{r'}$, Lemma~\ref{lem:dimension-reduction} and Theorem~\ref{thm:point-KL} give an $s$-subset $I\subseteq R$ such that $Q\cap K_I(\cC)\neq\varnothing$. Call such an $s$-subset \emph{good}, and let $\mathscr I$ be the set of all good $s$-subsets. We claim that $\abs{\mathscr I}\geq h+1$.

Indeed, otherwise choose one element from every member of $\mathscr I$, and enlarge the resulting set, if necessary, to an $h$-subset $H\subseteq[r]$. Every good set meets $H$, whereas $[r]\setminus H$ has size $r'$ and must contain a good $s$-subset, a contradiction.

Choose distinct good sets $I_1,\dots,I_h$, and choose an $h$-subset $H\subseteq[r]$ meeting each of them. Put $R_0=[r]\setminus H$. By Lemma~\ref{lem:dimension-reduction}, the family $\bigsqcup_{i\in R_0}\cC_i[Q]$ satisfies the condition for $(d',r',m',k')$. Apply the induction hypothesis in $\operatorname{aff}Q$ to obtain pairwise distinct $s$-subsets $J_1,\dots,J_{d'+1}\subseteq R_0$ such that $\sum_{a=1}^{d'+1}\dim\bigl(Q\cap K_{J_a}(\cC)\bigr)\geq0$.

Every $I_i$ meets $H$, while every $J_a$ is contained in $R_0$, so all the sets $I_1,\dots,I_h,J_1,\dots,J_{d'+1}$ are pairwise distinct. Since each $I_i$ is good, $\dim K_{I_i}(\cC)\geq0$, and $\dim K_{J_a}(\cC)\geq\dim\bigl(Q\cap K_{J_a}(\cC)\bigr)$. The total number of chosen sets is $h+d'+1=d+1$, and their dimension sum is nonnegative. This completes the induction.
\end{proof}

\section{Proof of Theorem~\ref{thm:top-dimensional}}\label{sec:top-dimensional}

\begin{proof}[Proof of Theorem~\ref{thm:top-dimensional}]
Since $M\subseteq X$, the contrapositive of Lemma~\ref{lem:KM-obstruction} gives a face $\tau\in X$ such that
\begin{equation}\label{eq:tau-KM}
\rho_0\coloneqq\rho_M(V\setminus\tau)
\leq\ell_X(\tau)\leq d.
\end{equation}
Choose an independent set $J\subseteq V\setminus\tau$ of size $\rho_0$ that spans $V\setminus\tau$, and extend $J$ to a basis $B$ of $M$. Put $I=B\setminus J$. For every $v\in I$, the set $B\setminus\{v\}$ contains $J$. Therefore,
\[
V\setminus\Span_M(B\setminus\{v\})
\subseteq V\setminus\Span_M(J)
\subseteq\tau.
\]
By Lemma~\ref{lem:delta-monotone} and Equation~\eqref{eq:tau-KM},
\[
\delta_X\paren{V\setminus\Span_M(B\setminus\{v\})}
\geq\delta_X(\tau)
\geq\ell_X(\tau)
\geq\rho_0.
\]
For every $v\in J$, the corresponding term is at least $-1$ by Definition~\ref{def:link-Leray-dimension}. Hence,
\begin{align*}
\sum_{v\in B}
\delta_X\paren{V\setminus\Span_M(B\setminus\{v\})}
&\geq\abs I\rho_0-\abs J\\
&=(r-\rho_0)\rho_0-\rho_0\\
&=(r-\rho_0-1)\rho_0\\
&\geq0,
\end{align*}
because $r\geq d+1\geq\rho_0+1$.
\end{proof}

\section{Proof of Theorem~\ref{thm:m=1ork>=m+d-1}}\label{sec:top-KL}

We first prove an auxiliary lemma. We then consider the three cases in Theorem~\ref{thm:m=1ork>=m+d-1}: $m=1$, $k=m+d-1$, and $k=m+d$.

\subsection{Auxiliary lemma}

\begin{lemma}\label{lem:certificate-selection}
Let $d\geq1$, $r\geq d+1$, and $k\geq2$ be integers, and let $1\leq q\leq r-k+1$. Let $X$ be a simplicial complex on $V$, and let $M$ be a rank-$r$ matroid on $V$. Suppose that there is a face $\tau\in X$ such that, with $\rho_0=\rho_M(V\setminus\tau)\leq k-1$, one has
\begin{equation}\label{eq:certificate}
(\delta_X(\tau)+1)
\binom{r-\rho_0}{q}
\geq d+1.
\end{equation}
Then there exist a basis $B$ of $M$ and $d+1$ pairwise distinct sets $I_1,\dots,I_{d+1}\in\binomset{B}{q}$
such that
\[
\sum_{i=1}^{d+1}
\delta_X(V\setminus\Span_M(B\setminus I_i))
\geq0.
\]
\end{lemma}

\begin{proof}
Choose an independent set $J\subseteq V\setminus\tau$ of size $\rho_0$ that spans $V\setminus\tau$, and extend $J$ to a basis $B$ of $M$. Put $I=B\setminus J$. Since $\rho_0\leq k-1$, we have $|I|=r-\rho_0\geq r-k+1\geq q$.
For every $I'\in\binomset{I}{q}$,
\[
V\setminus\Span_M(B\setminus I')
\subseteq V\setminus\Span_M(J)
\subseteq\tau.
\]
Therefore, Lemma~\ref{lem:delta-monotone} gives
\begin{equation}\label{eq:good-subset-bound}
\delta_X(V\setminus\Span_M(B\setminus I'))
\geq\delta_X(\tau).
\end{equation}

Set $N=\binom{r-\rho_0}{q}$. If $N\geq d+1$, choose any $d+1$ members of $\binomset{I}{q}$. Every term is nonnegative by Equation~\eqref{eq:good-subset-bound}, because $\tau$ is a face.

Assume that $N<d+1$. Choose all $N$ members of $\binomset{I}{q}$, and then choose $d+1-N$ pairwise distinct members of $\binomset{B}{q}\setminus\binomset{I}{q}$. Since $q\leq r-k+1\leq r-1$, we have
\[
\binom{r}{q}-N\geq r-N\geq d+1-N,
\]
which implies that there are enough such sets. Every additional term is at least $-1$. Thus, the total sum is at least
\[
N\delta_X(\tau)-(d+1-N)
=(\delta_X(\tau)+1)N-(d+1)
\geq0
\]
by Equation~\eqref{eq:certificate}.
\end{proof}

\subsection{The case \texorpdfstring{$m=1$}{m=1}}

We first prove a slightly stronger lemma, which will also be used later.

\begin{lemma}\label{lem:certificate-m1}
Let $d\geq1$, $r\geq d+1$, and $2\leq k\leq d+1$. Let $X$ be a $d$-Leray complex on $V$, and let $M$ be a rank-$r$ matroid on $V$. Assume that every $(d+1)$-set $U\subseteq V$ with $\rho_M(U)\geq k$ is a face of $X$.

Then, for every $q\in[r-k+1]$, there is a face $\tau\in X$ such that, with $\rho_0=\rho_M(V\setminus\tau)\leq k-1$,
\[
(\delta_X(\tau)+1)
\binom{r-\rho_0}{q}
\geq d+1.
\]
\end{lemma}

\begin{proof}
Suppose, to the contrary, that every face $\sigma\in X$ with $\rho_\sigma\coloneqq\rho_M(V\setminus\sigma)\leq k-1$ satisfies
\begin{equation}\label{eq:no-m1-certificate}
(\delta_X(\sigma)+1)
\binom{r-\rho_\sigma}{q}
\leq d.
\end{equation}
Since $\rho$. Set $t=d+1-k$, and consider the tolerance matroid $M^t$. We shall prove that for every $\sigma\in X$, we have $\rho^t(V\setminus\sigma) \geq\ell_X(\sigma)+1$.

Fix a face $\sigma\in X$. Assume first that $\rho_\sigma\geq k$. By Proposition~\ref{prop:tolerance-rank},
\[
\rho^t(V\setminus\sigma)
=\min\{|V\setminus\sigma|,\rho_\sigma+t\}
\geq\min\{|V\setminus\sigma|,d+1\}.
\]
If $|V\setminus\sigma|\geq d+1$, this is at least $\ell_X(\sigma)+1$, because $X$ is $d$-Leray. If $1\leq|V\setminus\sigma|\leq d$, then $\lk(X,\sigma)$ has at most $|V\setminus\sigma|$ vertices, and therefore
\[
\ell_X(\sigma)+1\leq|V\setminus\sigma|
=\rho^t(V\setminus\sigma).
\]

Assume now that $\rho_\sigma\leq k-1$. The set $V\setminus\sigma$ is nonempty. Indeed, if $\sigma=V$, then $\delta_X(\sigma)=0$, and the left-hand side of Equation~\eqref{eq:no-m1-certificate} is at least $\binom rq\geq r\geq d+1$, since $1\leq q\leq r-k+1\leq r-1$.

As above, $|V\setminus\sigma|\geq\ell_X(\sigma)+1$. Since
\[
\rho^t(V\setminus\sigma)
=\min\{|V\setminus\sigma|,\rho_\sigma+t\},
\]
it remains to prove that $\rho_\sigma+t\geq\ell_X(\sigma)+1$. Suppose, to the contrary, that
\[
\ell_X(\sigma)+1\geq\rho_\sigma+t+1.
\]
Set $a=k-1-\rho_\sigma$. Then $0\leq a\leq d$, and
\[
\rho_\sigma+t+1=d+1-a,
\qquad
q\leq r-k+1=r-\rho_\sigma-a.
\]

If $a=0$, then Lemma~\ref{lem:delta-monotone} gives
\[
(\delta_X(\sigma)+1)
\binom{r-\rho_\sigma}{q}
\geq\ell_X(\sigma)+1
\geq d+1,
\]
contradicting Equation~\eqref{eq:no-m1-certificate}.

Assume that $a\geq1$. Then $1\leq q\leq r-\rho_\sigma-1$. Moreover, $q\leq r-\rho_\sigma-a$ implies $r-\rho_\sigma\geq a+1$. Therefore,
\[
\binom{r-\rho_\sigma}{q}
\geq r-\rho_\sigma
\geq a+1.
\]
Together with $\delta_X(\sigma)\geq\ell_X(\sigma)$, this gives
\[
(\delta_X(\sigma)+1)
\binom{r-\rho_\sigma}{q}
\geq(d+1-a)(a+1)
=d+1+a(d-a)
\geq d+1,
\]
again contradicting Equation~\eqref{eq:no-m1-certificate}. This proves the claim.

By Lemma~\ref{lem:KM-obstruction}, $M^t\nsubseteq X$. Choose a set $\eta\in M^t\setminus X$. If $|\eta|\geq d+1$, the contrapositive of Lemma~\ref{lem:Leray-Helly}, together with the hereditary property of $X$, gives a $(d+1)$-subset $\eta'\subseteq\eta$ with $\eta'\notin X$. If $|\eta|<d+1$, extend $\eta$ in the matroid $M^t$ to an independent set $\eta'$ of size $d+1$. This is possible because Proposition~\ref{prop:tolerance-rank} gives $\rank(M^t)\geq\rank(M)=r\geq d+1$. Since $X$ is a simplicial complex and $\eta\notin X$, we also have $\eta'\notin X$.

In both cases, $\eta'\in M^t$, and hence $\rho_M(\eta')\geq|\eta'|-t=d+1-t=k$. The hypothesis of the lemma therefore gives $\eta'\in X$, a contradiction.
\end{proof}

\begin{proposition}[The case $m=1$]\label{prop:positive-m1}
Assume the hypothesis of Conjecture~\ref{conjecture: generalization of Kim-Lew's result} with $m=1$, and set $s=r-k+1$. Let $q\in[s]$, and define
\[
p=
\begin{cases}
1,&k=1\text{ and }q=r,\\
d+1,&\text{otherwise}.
\end{cases}
\]
Then there exist a basis $B$ of $M$ and pairwise distinct sets $I_1,\dots,I_p\in\binomset{B}{q}$
such that
\[
\sum_{i=1}^{p}
\delta_X(V\setminus\Span_M(B\setminus I_i))
\geq0.
\]
\end{proposition}

\begin{proof}
Suppose first that $k=1$. Let $S=V\setminus\Span_M(\varnothing)$ be the set of nonloops of $M$. Since $M$ has rank $r\geq d+1$, the set $S$ has at least $d+1$ vertices. Every $(d+1)$-subset of $S$ has positive rank and is therefore a face of $X$. Every subset of $S$ of size at most $d+1$ is consequently a face, so Lemma~\ref{lem:Leray-Helly} gives $S\in X$.

Let $B$ be any basis of $M$. For every nonempty set $I\subseteq B$, the monotonicity of the span operator gives
\[
V\setminus\Span_M(B\setminus I)
\subseteq V\setminus\Span_M(\varnothing)
=S.
\]
Hence, every corresponding link-Leray dimension is nonnegative. If $q=r$, take $I_1=B$. If $q<r$, then $\binom rq\geq r\geq d+1$, so we may choose any $d+1$ pairwise distinct $q$-subsets of $B$.

Assume now that $k\geq2$. Lemma~\ref{lem:certificate-m1}, followed by Lemma~\ref{lem:certificate-selection}, gives the result.
\end{proof}

\subsection{The case \texorpdfstring{$k=m+d-1$}{k=m+d-1}}

\begin{proposition}[The case $k=m+d-1$]\label{prop:positive-adjacent}
Assume the hypothesis of Conjecture~\ref{conjecture: generalization of Kim-Lew's result} and suppose that $k=m+d-1$. Set $s=r-k+1$. Let $q\in[s]$, and define
\[
p=
\begin{cases}
1,&k=1\text{ and }q=r,\\
d+1,&\text{otherwise}.
\end{cases}
\]
Then there exist a basis $B$ of $M$ and pairwise distinct sets $I_1,\dots,I_p\in\binomset{B}{q}$
such that
\[
\sum_{i=1}^{p}
\delta_X(V\setminus\Span_M(B\setminus I_i))
\geq0.
\]
\end{proposition}

\begin{proof}
We use induction on $m$. If $m=1$, then $k=d$, and the conclusion follows from Proposition~\ref{prop:positive-m1}.

Assume that $m\geq2$. Suppose first that the Kim--Lew condition in Conjecture~\ref{conjecture: generalization of Kim-Lew's result} also holds for the parameters $m-1$ and $k-1$. Since $k-1=(m-1)+d-1$, the induction hypothesis applies. Moreover,
\[
q\leq r-k+1<r-(k-1)+1.
\]
If $k-1=1$, then $q\leq r-1$, so the exceptional one-subset case in the induction hypothesis does not occur. Thus, the induction hypothesis gives the required $d+1$ subsets.

We may therefore assume that this condition fails for the parameters $m-1$ and $k-1$. Hence, there is a set
\[
U=A\sqcup\{v_1,\dots,v_{m-2}\},
\qquad |A|=d+1,
\]
such that $\rho_M(U)\geq k-1$, $A\notin X$, and, for every $j\in[m-2]$, there is a proper subset $\sigma_j\subsetneq A$ satisfying $\sigma_j\cup\{v_j\}\notin X$. Since $|U|=d+m-1=k$, either $\rho_M(U)=k-1$ or $\rho_M(U)=k$.

\smallskip
\noindent\emph{Case 1: $\rho_M(U)=k-1$.}
Put $W=V\setminus\Span_M(U)$. Since $\rho_M(U)=k-1<r$, the set $U$ does not span $V$, and hence $W\neq\varnothing$. For every $w\in W$, the set $U\cup\{w\}$ has rank $k$. Apply the original Kim--Lew condition to
\[
A\sqcup\{v_1,\dots,v_{m-2},w\}.
\]
The set $A$ is not a face, and each $v_j$ has a proper subset $\sigma_j\subsetneq A$ for which $\sigma_j\cup\{v_j\}\notin X$. Therefore, the condition forces
\[
\sigma\cup\{w\}\in X
\qquad\text{for every proper subset }\sigma\subsetneq A.
\]

Since $W\neq\varnothing$, by Lemma~\ref{lem:KL-boundary}, we have
\[
W\cup(A\setminus\{a\})\in X
\qquad\text{for every }a\in A.
\]
In particular, $W\in X$. The induced subcomplex $\lk(X,W)[A]$ is the boundary of the $d$-simplex on $A$: every proper subset of $A$ belongs to the link, whereas $A$ does not because $A\notin X$. Hence, $\delta_X(W)\geq d$. Moreover,
\[
\rho_M(V\setminus W)
=\rho_M(\Span_M(U))
=k-1.
\]
Therefore,
\[
(\delta_X(W)+1)
\binom{r-(k-1)}{q}
\geq(d+1)\binom{s}{q}
\geq d+1.
\]
By Lemma~\ref{lem:certificate-selection}, we are done.

\smallskip
\noindent\emph{Case 2: $\rho_M(U)=k$.}
Since $|U|=k$, the set $U$ is independent. Put $T=V\setminus U$. The assumption $|V|\geq m+d=k+1$ gives $T\neq\varnothing$. For every $w\in T$, the set $U\cup\{w\}$ has rank at least $k$. As in Case~1, the original Kim--Lew condition gives
\[
\sigma\cup\{w\}\in X
\qquad\text{for every proper subset }\sigma\subsetneq A.
\]

Since $T\neq\varnothing$, by Lemma~\ref{lem:KL-boundary}, we have
\[
T\cup(A\setminus\{a\})\in X
\qquad\text{for every }a\in A.
\]
Since $d\geq1$, for each $a\in A$ we may choose $b\in A\setminus\{a\}$. Then
\[
T\cup\{a\}\subseteq T\cup(A\setminus\{b\})\in X,
\]
so $T\cup\{a\}$ is a face for every $a\in A$.

Extend the independent set $U$ to a basis $B$ of $M$. Since $|B\setminus U|=r-k=s-1$ and $q\leq s$, choose a set $R\subseteq B\setminus U$ of size $q-1$. For every $a\in A$, put $I_a=R\cup\{a\}$. Since $R\cap A=\varnothing$, these are $d+1$ pairwise distinct $q$-subsets of $B$. Moreover, $B\setminus I_a$ contains $U\setminus\{a\}$, so
\[
V\setminus\Span_M(B\setminus I_a)
\subseteq V\setminus(U\setminus\{a\})
=T\cup\{a\}.
\]
Thus, every set on the left is a face of $X$, and all $d+1$ terms are nonnegative. This completes the proof.
\end{proof}

\subsection{The case \texorpdfstring{$k=m+d$}{k=m+d}}

\begin{lemma}\label{lem:certificate-extreme}
Assume the hypothesis of Conjecture~\ref{conjecture: generalization of Kim-Lew's result} and suppose that $k=m+d$. Set $s=r-k+1$. Then, for every $q\in[s]$, there is a face $\tau\in X$ such that, with $\rho_0=\rho_M(V\setminus\tau)\leq k-1$, one has
\[
(\delta_X(\tau)+1)
\binom{r-\rho_0}{q}
\geq d+1.
\]
\end{lemma}

\begin{proof}
We use induction on $m$. If $m=1$, then $k=d+1$, and the result follows from Lemma~\ref{lem:certificate-m1}. Assume that $m\geq2$. Suppose first that the Kim--Lew condition also holds for the parameters $m-1$ and $k-1=(m-1)+d$. Since $q\leq r-k+1<r-(k-1)+1$, the induction hypothesis gives the required face.

We may therefore assume that the smaller-parameter condition fails. Then there is a set
\[
U=A\sqcup\{v_1,\dots,v_{m-2}\},
\qquad |A|=d+1,
\]
such that $\rho_M(U)\geq k-1$, $A\notin X$, and, for every $j\in[m-2]$, there is a proper subset $\sigma_j\subsetneq A$ satisfying $\sigma_j\cup\{v_j\}\notin X$.
Indeed, $|U|=d+m-1=k-1$, so the rank threshold in the failed condition forces $\rho_M(U)=k-1$.

The argument in Case~1 of the proof of Proposition~\ref{prop:positive-adjacent} applies verbatim and shows that $W=V\setminus\Span_M(U)$ is a face of $X$ with $\delta_X(W)\geq d$ and $\rho_M(V\setminus W)=k-1$. Since $1\leq q\leq s=r-k+1$, taking $\tau=W$ completes the proof.
\end{proof}

\begin{proposition}[The case $k=m+d$]\label{prop:positive-extreme}
Assume the hypothesis of Conjecture~\ref{conjecture: generalization of Kim-Lew's result} and suppose that $k=m+d$. Set $s=r-k+1$. Then there exist a basis $B$ of $M$ and $d+1$ pairwise distinct sets $I_1,\dots,I_{d+1}\in\binomset{B}{s}$ such that
\[
\sum_{i=1}^{d+1}
\delta_X(V\setminus\Span_M(B\setminus I_i))
\geq0.
\]
\end{proposition}

\begin{proof}
Apply Lemma~\ref{lem:certificate-extreme} with $q=s$, and then apply Lemma~\ref{lem:certificate-selection}.
\end{proof}

Propositions~\ref{prop:positive-m1},~\ref{prop:positive-adjacent}, and~\ref{prop:positive-extreme} prove Theorem~\ref{thm:m=1ork>=m+d-1}.

\section{Counterexample and sharpness}\label{sec:counterexample}

We now prove Theorem~\ref{theorem:counterexample-and-sharpness}. The counterexample and the sharpness construction arise from the same family of simplicial complexes.

For a finite set $E$, write
\[
K_d(E)=\{A\subseteq E:|A|\leq d\},
\]
which is the $(d-1)$-skeleton of the simplex on $E$. If $E$ and $Q$ are disjoint, define
\[
X_d(E,Q)=K_d(E)*2^Q
=\{S\subseteq E\sqcup Q:|S\cap E|\leq d\}.
\]

\begin{observation}\label{lem:common-construction}
Let $|E|=n\geq d+1$, $|Q|=b\geq1$, and $1\leq a\leq n$. Set $X=X_d(E,Q)$ and $M=U_{a,E}\oplus U_{b,Q}$. Then the following statements hold.
\begin{enumerate}[label=\textup{(\roman*)}]
\item The complex $X$ is $d$-Leray, and $L(K_d(E))=d$.
\item Every basis of $M$ has the form $B=D\sqcup Q$, where $D\in\binomset{E}{a}$. For $I\subseteq B$, put $x=|I\cap D|$ and $S_I=V\setminus\Span_M(B\setminus I)$.
Then
\begin{equation}\label{eq:span-computation}
S_I=
\begin{cases}
I,&x=0,\\[1mm]
(E\setminus(D\setminus I))\sqcup(I\cap Q),&x>0.
\end{cases}
\end{equation}
\item Under the notation of part~\textup{(ii)}, if $n=d+a$ and $|I|=b$, then
\begin{equation}\label{eq:one-good}
\delta_X(S_I)=
\begin{cases}
d,&I=Q,\\
-1,&I\neq Q.
\end{cases}
\end{equation}
\end{enumerate}
\end{observation}

\begin{proof}
For part~\textup{(i)}, let $W\subseteq E\sqcup Q$. If $W\cap Q\neq\varnothing$, then $X[W]=K_d(E\cap W)*2^{Q\cap W}$ is a cone. If $W\cap Q=\varnothing$, then $X[W]$ has dimension at most $d-1$. Thus, every induced subcomplex of $X$ has vanishing reduced homology in dimensions at least $d$, and $X$ is $d$-Leray.

The complex $K_d(E)$ has dimension $d-1$, so its Leray number is at most $d$. On the other hand, its restriction to any $(d+1)$-subset of $E$ is the boundary of a $d$-simplex and has nonzero reduced homology in dimension $d-1$. Therefore, $L(K_d(E))=d$.

For part~\textup{(ii)}, if $x=0$, then $I\subseteq Q$. The set $D$ spans all of $E$, while $Q\setminus I$ spans itself. Hence,
\[
\Span_M(B\setminus I)=E\sqcup(Q\setminus I),
\]
which implies $S_I=I$. If $x>0$, then $|D\setminus I|<a$, so the span of $D\setminus I$ in $U_{a,E}$ is $D\setminus I$. The $Q$-part is again $Q\setminus I$, and Equation~\eqref{eq:span-computation} follows.

For part~\textup{(iii)}, if $I=Q$, then $S_I=Q$ and $\lk(X,Q)=K_d(E)$. Part~\textup{(i)} gives $\delta_X(S_I)=d$. Assume that $I\neq Q$. Since $I\subseteq B=D\sqcup Q$ and $|I|=|Q|=b$, we must have $x=|I\cap D|\geq1$. Equation~\eqref{eq:span-computation} now gives
\[
|S_I\cap E|
=n-|D\setminus I|
=d+a-(a-x)
=d+x>d.
\]
Thus, $S_I\notin X$, and $\delta_X(S_I)=-1$.
\end{proof}

\begin{proof}[Proof of Theorem~\ref{theorem:counterexample-and-sharpness}]
We first prove part~\textup{(i)}. Let $E$ be a set of size $d+1$, let $Q=\{q\}$, and set
\[
X=X_d(E,Q),
\qquad
M=U_{d,E}\oplus U_{1,Q}.
\]
By Observation~\ref{lem:common-construction}, the complex $X$ is $d$-Leray, and the matroid $M$ has rank $d+1$. The only nonface of $X$ of size $d+1$ is $E$. Since $V$ is the only subset of $V$ of size $d+2$, Equation~\eqref{eq:strong-KL-condition} holds for $m=k=2$.

Every basis of $M$ has the form $B=D\sqcup\{q\}$, where $D\in\binomset{E}{d}$. Since $r-k+1=d$, the relevant subsets are the $d$-subsets of $B$. For $I=D$, we have
\[
V\setminus\Span_M(B\setminus I)=E\notin X,
\]
so the corresponding term is $-1$.

Every other $d$-subset has the form $I=B\setminus\{x\}$ for some $x\in D$. Since $d\geq2$, the singleton $\{x\}$ spans only itself in $U_{d,E}$. Therefore, $V\setminus\Span_M(B\setminus I)=V\setminus\{x\}$. This set is a facet of $X$. Its link is $\{\varnothing\}$ and has Leray number $0$. Thus, the remaining $d$ terms are $0$, and Equation~\eqref{eq:counterexample-total} follows.

We next prove part~\textup{(ii)}. Put $s=r-k+1$, choose disjoint sets $E$ and $Q$ with
\[
|E|=d+k-1,
\qquad
|Q|=s,
\]
and define
\[
X=X_d(E,Q),
\qquad
M=U_{k-1,E}\oplus U_{s,Q}.
\]
By Observation~\ref{lem:common-construction}, the complex $X$ is $d$-Leray, and the matroid $M$ has rank
\[
(k-1)+s=r.
\]
For every basis $B$ of $M$, Equation~\eqref{eq:one-good} shows that exactly one member of $\binomset{B}{s}$, namely $Q$, has link-Leray dimension $d$, while every other member has value $-1$. Thus, the sum over any $\ell$ pairwise distinct members is at most
\[
d-(\ell-1)=d+1-\ell,
\]
and equality is attained by choosing $Q$ together with any $\ell-1$ other members. This proves Equation~\eqref{eq:exact-sharpness-profile}.

It remains to verify the two further statements. Note first that
\[
|V|=|E|+|Q|=d+r,
\]
so the size requirement $|V|\geq\max\{m+d,r\}$ holds for every $m\leq r$.

Suppose first that $2\leq k\leq d+1$ and $m=1$. A $(d+1)$-set that is not a face of $X$ must be contained in $E$, and hence has rank at most $k-1$. Therefore, every $(d+1)$-set of rank at least $k$ is a face of $X$. This is precisely the Kim--Lew condition in Conjecture~\ref{conjecture: generalization of Kim-Lew's result} when $m=1$.

Assume now that $m$ lies in the parameter range of Conjecture~\ref{conjecture: generalization of Kim-Lew's result} and that $k\geq m+d-1$. Consider a set
\[
U=\{u_1,\dots,u_{d+1},v_1,\dots,v_{m-1}\}
\]
with $\rho_M(U)\geq k$, and put $A=\{u_1,\dots,u_{d+1}\}$. If $A\in X$, there is nothing to prove. Assume that $A\notin X$. Then $A\subseteq E$. If $m=1$, then $U=A$ and $\rho_M(U)\leq k-1$, a contradiction. Hence, the Kim--Lew condition holds in this case.

Assume that $m\geq2$. If all vertices $v_1,\dots,v_{m-1}$ belonged to $E$, then $U\subseteq E$ and $\rho_M(U)\leq k-1$, again a contradiction. Thus, some $v_j$ belongs to $Q$. Every proper subset $\sigma\subsetneq A$ has at most $d$ vertices in $E$, and therefore $\sigma\cup\{v_j\}\in X$. This proves the Kim--Lew condition for every $m$ in the stated range and completes the proof.
\end{proof}

\section*{Use of Generative-AI tools declaration}
ChatGPT was used in developing the proofs and writing this paper. In particular, discussions with ChatGPT helped improve and unify the results presented in Theorem~\ref{theorem:counterexample-and-sharpness}. It was also used extensively to refine the logical structure of the arguments and polish the exposition. The author has verified all mathematical claims and takes full responsibility for the final content.

\section*{Acknowledgments}
The author thanks Alexander Polyanskii for introducing the problem, suggesting a possible approach to the geometric version of the dimensional colorful Helly theorem (Theorem~\ref{thm:dim-colorful}), and proposing the link-Leray dimension in Definition~\ref{def:link-Leray-dimension} as an analogue of geometric dimension for simplicial complexes. The author is also grateful to him and Minki Kim for fruitful discussions.

\bibliographystyle{alpha}
\bibliography{references}

\end{document}